\documentclass[joc]{ipart}

\usepackage{amsthm,amsmath,amssymb}
\RequirePackage{hyperref}

\pubyear{2014} \volume{0} \issue{0} \firstpage{1} \lastpage{1}
\arxiv{http://arxiv.org/abs/1304.3154}

\startlocaldefs
\newtheorem{theorem}{Theorem}[section]
\endlocaldefs

\begin{document}

\begin{frontmatter}

\title{An infinite cardinal version  of Gallai's Theorem for colorings of the plane}
\runtitle{Infinite cardinal Gallai}


\author{\fnms{Jeremy F.} \snm{Alm}\ead[label=e1]{alm.academic@gmail.com}}
\address{Department of Mathematics\\ Illinois College\\ Jacksonville, IL 62650\\ \printead{e1}}

\runauthor{J.~F.~Alm}

\renewcommand{\thefootnote}{\fnsymbol{footnote}}
\footnotetext{2010 MSC codes 05D10, 52C10, 05C50}
\renewcommand{\thefootnote}{\arabic{footnote}}

\begin{abstract}
We generalize a result of Tibor Gallai as follows: for any finite
set of points $\mathcal{S}$ in the plane,  if the plane is colored
in finitely many colors, then there exist $2^{\aleph_0}$
monochromatic subsets of the plane homothetic to $\mathcal{S}$.
Furthermore, we prove an even stronger result for $n$-dimensional
Euclidean space.

\end{abstract}


\end{frontmatter}

\section{Introduction}
\subsection{History}
In the late 1930s, Tibor  Gallai proved the following result, which
is a generalization to multiple dimensions of Van der Waerden's
celebrated theorem on arithmetic progressions in the integers:

\begin{theorem}[Gallai's Theorem on $\mathbb{Z}^n$] \label{GZ}
   Let $\mathcal{S}$ be a finite subset of $\mathbb{Z}^n$.  Then any finite coloring of $\mathbb{Z}^n$ contains a monochromatic subset homothetic to $\mathcal{S}$, i.e. a monochromatic subset of the form $\mathbf{a}+b\mathcal{S}$ for some $\mathbf{a}\in\mathbb{Z}^n$ and $b\in\mathbb{Z}$.
 \end{theorem}
Gallai himself did not publish the result; a proof  first appeared
in the literature in \cite{Rado}.  (For an interesting history, see
Chapter 42 of Soifer's book \cite{Soifer}.) A derivation of Theorem
\ref{GZ} from the Hales-Jewett Theorem  can be found in
\cite{GrahamRT}. A Euclidean-space-version of Gallai's Theorem is as
follows:

 \begin{theorem}[Gallai's Theorem on $\mathbb{E}^n$] \label{GE}
   Let $\mathcal{S}$ be any finite subset of $\mathbb{E}^n$.  Then any finite coloring of $\mathbb{E}^n$ contains a monochromatic subset homothetic to $\mathcal{S}$, i.e. a monochromatic subset of the form $\mathbf{a}+b\mathcal{S}$ for $\mathbf{a}\in\mathbb{E}^n$ and $b\in\mathbb{R}$.  \end{theorem}

   These two theorems are fundamental in   Euclidean Ramsey theory. It seems natural, however, to suspect that Theorem \ref{GE} can be strengthened.  Surely, we can find more than \emph{one} monochromatic homothetic copy, no? In this paper, we generalize  Theorem \ref{GE} by showing that there must in fact be uncountably many monochromatic homothetic copies.

   To the best of our knowledge, the first uncountability result of this type appeared in \cite{BDP}:

   \begin{theorem} \label{BDP}
     For every 2-coloring of $\mathbb{E}^2$, there exist uncountably many
$r \in \mathbb{R}^+$ such that there exists a monochromatic
equilateral triangle of side length $r$.
   \end{theorem}

   The author's first attempt to strengthen Gallai's Theorem resulted in the following two theorems  from \cite{Alm2012G}, the second of which strengthens Theorem \ref{BDP}:

   \begin{theorem}
    Let $\mathcal{S}$ be a finite configuration of points in the integer lattice $\mathbb{Z}^2$.  In any finite coloring of the plane $\mathbb{E}^2$, there exist $2^{\aleph_0}$ monochromatic homothetic copies of $\mathcal{S}$.
 \end{theorem}

   \begin{theorem} \label{3pt}
   For any finite coloring of $\mathbb{E}^2$ there exist $2^{\aleph_0}$  monochromatic equilateral triangles, i.e., subsets homothetic to $\{(0,0), (1,0), (1/2,\sqrt{3}/2)\}$.
 \end{theorem}

 In fact, the argument given in \cite{Alm2012G} for Theorem \ref{3pt} will work for any 3-point configuration in $\mathbb{E}^2$. The problem for an arbitrary finite subset $\mathcal{S}\subset\mathbb{E}^2$ was left open.  The main result of the present paper is a solution to that problem:

 \begin{theorem} \label{Main}
  Let $\mathcal{S}$ be any finite set of points in $\mathbb{E}^2$.  If the points of $\mathbb{E}^2$ are colored in finitely many colors, then there exist $2^{\aleph_0}$ pairwise-disjoint monochromatic subsets of $\mathbb{E}^2$ homothetic to $\mathcal{S}$.
\end{theorem}

\subsection{Notation}

For a positive integer $m$, we let $[m]$ denote
$\{1,2,3,\ldots,m\}$.  We let $\mathbb{E}^n$ denote $n$-dimensional
Euclidean space, and $\mathbb{R}$ the real line.  For $n>1$,
elements of $\mathbb{E}^n$ will be denoted with boldface, as in
$\textbf{y}$, and $\textbf{0}$ will denote the origin.  Subsets of
$\mathbb{R}$ and $\mathbb{E}^n$ will be denoted with capital script
(``mathcal") letters.  For $\mathcal{X}\subseteq\mathbb{E}^n$,
$a\in\mathbb{R}$, and $\textbf{d}\in\mathbb{E}^n$, let
$a\mathcal{X}+\textbf{d}=\{a\textbf{x}+\textbf{d}:\textbf{x}\in\mathcal{X}\}$.

Recall that for all $n\in\mathbb{Z}^+$, the cardinality of
$\mathbb{E}^n$ is $2^{\aleph_0}$.

\section{Main Results}

We will prove the following two results, whose conjunction is
equivalent to Theorem \ref{Main}.

\begin{theorem} \label{notline}
  Let $\mathcal{S}$ be any finite set of points in $\mathbb{E}^2$, not all on a line.  If the points of $\mathbb{E}^2$ are colored in finitely many colors, then there exist $2^{\aleph_0}$ pairwise-disjoint  monochromatic subsets of $\mathbb{E}^2$ homothetic to $\mathcal{S}$.
\end{theorem}

\begin{theorem} \label{line}
  Let $\mathcal{S}$ be any finite set of points in $\mathbb{R}$.  If the points of $\mathbb{R}$ are colored in finitely many colors, then there exist $2^{\aleph_0}$  pairwise-disjoint monochromatic subsets of $\mathbb{R}$ homothetic to $\mathcal{S}$.
\end{theorem}

\begin{proof}[Proof of Theorem \ref{notline}]
  Let $\mathcal{S}\subset\mathbb{E}^2$, with $|\mathcal{S}|=n$.  Assume without loss of generality that $\mathcal{S}$ contains the origin $\textbf{0}$, so that $\mathcal{S}=\{\textbf{0},\textbf{y}_1,\ldots,\textbf{y}_{n-1}\}$.  Let $A$ be a 2-by-$(n-1)$ matrix whose $j^{\text{th}}$ column is $\textbf{y}_j$ (considered as a 2-by-1 column vector):
  $$A:=[\textbf{y}_1\textbf{y}_2\cdots\textbf{y}_{n-1}].$$
  Let $T:\mathbb{E}^{n-1}\rightarrow\mathbb{E}^2$ be given by $T(\textbf{v})=A\textbf{v}$.  Note that since the points of $\mathcal{S}$ do not lie all on a line, the set $\{\textbf{y}_1,\ldots,\textbf{y}_{n-1}\}$ spans $\mathbb{E}^2$ and hence $T$ is surjective.

  Let $c\in\mathbb{Z}^+$, and let $\chi:\mathbb{E}^2\rightarrow[c]$ be a coloring of the plane with $c$ colors.  Define $\chi':\mathbb{E}^{n-1}\rightarrow[c]$ by $\chi'(\mathbf{v})=\chi(T(\mathbf{v}))$.  Let $\{\textbf{u}_1,\ldots, \textbf{u}_{n-1}\}$ be the standard basis for $\mathbb{E}^{n-1}$, and let $\mathcal{U}=\{\textbf{0},\textbf{u}_1,\ldots,\textbf{u}_{n-1}\}$.

  We must now partition $\mathbb{E}^{n-1}$ into $2^{\aleph_0}$ parts.  Consider the collection of cosets ${\mathbb{E}^{n-1}}/{\mathbb{Z}^{n-1}}$.  Each coset $\mathcal{C}\in {\mathbb{E}^{n-1}}/{\mathbb{Z}^{n-1}}$ is of the form $\mathbb{Z}^{n-1}+\textbf{e}$ for some $\textbf{e}\in[0,1)^{n-1}$.  Hence $\mathcal{C}$ is a translation of $\mathbb{Z}^{n-1}$.  It will prove useful later in the proof to index the cosets by translation vectors from $[0,1)^{n-1}$, as in $\mathcal{C}_\textbf{e}$.  Note that each $\textbf{e}\in[0,1)^{n-1}$ uniquely identifies a coset.

  Each coset $\mathcal{C}_\textbf{e}$ is colored by $\chi'$, so by Gallai's Theorem on $\mathbb{Z}^n$ we may conclude that there exists some monochromatic subset of $\mathcal{C}_\textbf{e}$ homothetic to $\mathcal{U}$.  This subset will have the form $a\mathcal{U}+\textbf{d}+\textbf{e}$, where $a\in\mathbb{Z}^+$ and $\textbf{d}\in\mathbb{Z}^{n-1}$.  Now define
  \begin{equation*}
  \mathcal{S}_\textbf{e}=T(a\mathcal{U}+\textbf{d}+\textbf{e})=\{T(\textbf{v}):\textbf{v}\in a\mathcal{U}+\textbf{d}+\textbf{e}\}.
  \end{equation*}
  It is clear that $\mathcal{S}_\textbf{e}$ is monochromatic (under $\chi$).  To see that $\mathcal{S}_\textbf{e}$ is homothetic to $\mathcal{S}$, we recall that $T$ was defined by $T(\textbf{v})=A\textbf{v}$, where the columns of $A$ were the $\textbf{y}_i$'s, so $T$ sends the unit basis vectors in $\mathcal{U}$ to the $\textbf{y}_i$'s in $\mathcal{S}$.  Hence $T(\mathcal{U})=\mathcal{S}$, and we have
  \begin{equation*}
    T(a\mathcal{U}+\textbf{d}+\textbf{e})=aT(\mathcal{U})+T(\textbf{d}+\textbf{e})=a\mathcal{S}+T(\textbf{d}+\textbf{e}).
  \end{equation*}
  Therefore $\mathcal{S}_\textbf{e}$ is homothetic to $\mathcal{S}$.  Since there are $2^{\aleph_0}$ cosets, we have $2^{\aleph_0}$ monochromatic homothetic copics of $\mathcal{S}$.

  It remains to show that, while many of these $\mathcal{S}_\textbf{e}$ may overlap, it is possible to find $2^{\aleph_0}$ of them that are pairwise-disjoint.

  Since $A$ has rank 2, the image under $T$ of $[0,1)^{n-1}$ contains a parallelogram $\mathcal{P}$.  Let $\mathcal{Y}=\{m(\textbf{y}_i-\textbf{y}_j):m\in\mathbb{Z},\,\, 1\leq i\leq j<n\}$, and define an equivalence relation $\sim$ on $\mathcal{P}$ as follows: $$\textbf{p}\sim \textbf{q} \text{ if and only if }\textbf{p}-\textbf{q}\in\mathcal{Y}.$$
  Let $\mathcal{P}_0$ be a subset of $\mathcal{P}$ formed by choosing exactly one representative from each equivalence class.  Since each equivalence class is at  most countable, $|\mathcal{P}_0|=2^{\aleph_0}$.  For each $\textbf{p}\in\mathcal{P}_0$, choose $\textbf{e}_\textbf{p}\in[0,1)^{n-1}$ such that $T(\textbf{e}_\textbf{p})=\textbf{p}$, and let $\mathcal{E}=\{\textbf{e}_\textbf{p}:\textbf{p}\in\mathcal{P}_0\}$.  Note that $T$ is 1-1 on $\mathcal{E}$, and that $|\mathcal{E}|=|\mathcal{P}_0|=2^{\aleph_0}$.

  Now for each $\textbf{e}_\textbf{p}\in\mathcal{E}$ there is a coset $\mathcal{C}_{\textbf{e}_\textbf{p}}$ that contains a monochromatic set homothetic to $\mathcal{U}$ of the form $a\mathcal{U}+\textbf{d}+\textbf{e}_\textbf{p}$, where $a\in\mathbb{Z}^+$ and $\textbf{d}\in\mathbb{Z}^{n-1}$.  Since $\mathbb{Z}^+$ and $\mathbb{Z}^{n-1}$ are countable, there must be some $a_0\in\mathbb{Z}^+$ and $\textbf{d}_0\in\mathbb{Z}^{n-1}$ such that for $2^{\aleph_0}$ of the cosets $\mathcal{C}_{\textbf{e}_\textbf{p}}$, the monochromatic subset of $\mathcal{C}_{\textbf{e}_\textbf{p}}$ homothetic to $\mathcal{U}$ has the form $a_0\mathcal{U}+\textbf{d}_0+\textbf{e}_\textbf{p}$.

  Let $\mathcal{E}'$ be the subset of $\mathcal{E}$ consisting of all $\textbf{e}_\textbf{p}$ such that $a_0\mathcal{U}+\textbf{d}_0+\textbf{e}_\textbf{p}$ is monochromatic; note that $|\mathcal{E}'|=2^{\aleph_0}$.  We show that for any two distinct $\textbf{e}_\textbf{p}$, $\textbf{e}_\textbf{q}\in\mathcal{E}'$, $\mathcal{S}_{\textbf{e}_\textbf{p}}$ and $\mathcal{S}_{\textbf{e}_\textbf{q}}$ are disjoint.

  Now
  \begin{equation*}
    \mathcal{S}_{\textbf{e}_\textbf{p}}=T(a_0\mathcal{U}+\textbf{d}_0+\textbf{e}_\textbf{p})=a_0\mathcal{S}+T(\textbf{d}_0)+
    \textbf{p}
  \end{equation*}
  and
  \begin{equation*}
    \mathcal{S}_{\textbf{e}_\textbf{q}}=T(a_0\mathcal{U}+\textbf{d}_0+\textbf{e}_\textbf{q})=a_0\mathcal{S}+T(\textbf{d}_0)+\textbf{q}.
  \end{equation*}
  Suppose there is some point  common to both sets.  Then there are $\textbf{y}_i$, $\textbf{y}_j\in\mathcal{S}$ such that
  \begin{equation*}
    a_0\textbf{y}_i+T(\textbf{d}_0)+\textbf{p}=a_0\textbf{y}_j+T(\textbf{d}_0)+\textbf{q}.
  \end{equation*}
  Rearranging yields
  \begin{equation*}
    a_0(\textbf{y}_i-\textbf{y}_j)=\textbf{q}-\textbf{p}.
  \end{equation*}
  Hence $\textbf{p}\sim\textbf{q}$, a contradiction.

  Therefore the sets $\mathcal{S}_{\textbf{e}_\textbf{p}}$, for $\textbf{e}_\textbf{p}\in\mathcal{E}'$, are pairwise disjoint.  Since $|\mathcal{E}'|=2^{\aleph_0}$, the proof is complete.
\end{proof}

The next proof will run  in parallel to the previous one.


\begin{proof}[Proof of Theorem \ref{line}]

Let $\mathcal{S}\subset\mathbb{R}$, with $|\mathcal{S}|=n$.  Assume
without loss of generality that 0 is the smallest number in
$\mathcal{S}$, so that $\mathcal{S}$ contains $0<x_1<x_2<\ldots
<x_{n-1}$.  Let $A$ be a 1-by-$(n-1)$ row vector
$$A:=[x_1x_2\cdots x_{n-1}].$$
Let $T:\mathbb{E}^{n-1}\rightarrow \mathbb{R}$ be given by
$T(\textbf{v})=A\textbf{v}$.

Let $c\in\mathbb{Z}^+$, and let $\chi:\mathbb{R}\rightarrow[c]$ be a
coloring of the real line with $c$ colors.  Define
$\chi':\mathbb{E}^{n-1}\rightarrow[c]$ by
$\chi'(\textbf{v})=\chi(T(\textbf{v}))$.  Let
$\{\textbf{u}_1,\ldots,\textbf{u}_{n-1}\}$ be the standard basis for
$\mathbb{E}^{n-1}$, and let
$\mathcal{U}=\{\textbf{0},\textbf{u}_1,\ldots,\textbf{u}_{n-1}\}$.

As in the proof of Theorem \ref{notline}, we have for each
$\textbf{e}\in[0,1)^{n-1}$ a monochromatic subset of
$a\mathcal{U}+\textbf{d}+\textbf{e}$ and its image
$\mathcal{S}_\textbf{e}$ under $T$.  The image of $[0,1)^{n-1}$
under $T$ contains an interval $(a,b)$.  Let
$\mathcal{Y}=\{m(x_i-x_j):m\in\mathbb{Z},\; 1\leq i\leq j<n\}$, and
define an equivalence relation $\sim$ on $(a,b)$ as follows:
$$p\sim q\text{ if and only if }p-q\in\mathcal{Y}.$$
Let $(a,b)^\star$ be a subset of $(a,b)$ containing exactly one
representative from each equivalence class.  Define $\textbf{e}_p$,
$\mathcal{E}$, and $\mathcal{E}'$ as in the proof  of  Theorem
\ref{notline}.  Then by the same argument given at the end of that
proof, we have that for distinct $\textbf{e}_p$ and
$\textbf{e}_q\in\mathcal{E}'$, $\mathcal{S}_{\textbf{e}_p}$ and
$\mathcal{S}_{\textbf{e}_q}$ are disjoint.

\end{proof}

It is not hard to see how to extend this argument to $k$-dimensional
Euclidean space.  We have the following:

\begin{theorem} \label{noproof}
  Let $n$, $k\in\mathbb{Z}^+$, with $n>k$.  Let $\mathcal{S}$ be an $n$-element subset of $\mathbb{E}^k$, whose points are not all contained in any $(k-1)$-dimensional hyperplane.  If the points of $\mathbb{E}^k$ are colored in finitely many colors, then there exist $2^{\aleph_0}$ pairwise-disjoint  monochromatic subsets homothetic to $\mathcal{S}$.
\end{theorem}

Instead of proving Theorem \ref{noproof}, we will state and prove an
even stronger version.  Note that in the proof of Theorem
\ref{notline}, the sets $\mathcal{S}_{\textbf{e}_\textbf{p}}$ all
had the form $a_0\mathcal{S}+\textbf{d}+\textbf{e}$, where $a_0$ was
a positive integer.  Let us call the multiplicative constant $a_0$
the \emph{dilation factor}.  We can use the fact that these dilation
factors are integer-valued to get a stronger result.

\begin{theorem} \label{big}
 Let $n$, $k\in\mathbb{Z}^+$, with $n>k$.  Let $\mathcal{S}$ be an $n$-element subset of $\mathbb{E}^k$, whose points are not all contained in any $(k-1)$-dimensional hyperplane.  If the points of $\mathbb{E}^k$ are colored in finitely many colors, then there exist $2^{\aleph_0}$ positive numbers $r$ such that for each such $r$, there exist $2^{\aleph_0}$ pairwise-disjoint  monochromatic subsets homothetic to $\mathcal{S}$ with dilation factor $r$.
\end{theorem}


\begin{proof}[Proof of Theorem \ref{big}]
  Let $\mathcal{S}\subset\mathbb{E}^k$, with $|\mathcal{S}|=n$.  Assume without loss of generality that $\mathcal{S}$ contains the origin $\textbf{0}$, so that $\mathcal{S}=\{\textbf{0},\textbf{y}_1,\ldots,\textbf{y}_{n-1}\}$.  Let $A$ be a $k$-by-$(n-1)$ matrix whose $j^{\text{th}}$ column is $\textbf{y}_j$ (considered as a $k$-by-1 column vector):
  $$A:=[\textbf{y}_1\textbf{y}_2\cdots\textbf{y}_{n-1}].$$
  Let $T:\mathbb{E}^{n-1}\rightarrow\mathbb{E}^k$ be given by $T(\textbf{v})=A\textbf{v}$.  Note that since the points of $\mathcal{S}$ do not lie on any $(k-1)$-dimensional hyperplane, the set $\{\textbf{y}_1,\ldots,\textbf{y}_{n-1}\}$ spans $\mathbb{E}^k$ and hence $T$ is surjective.

  Let $c\in\mathbb{Z}^+$, and let $\chi:\mathbb{E}^k\rightarrow[c]$ be a coloring of $\mathbb{E}^k$ with $c$ colors.  Define $\chi':\mathbb{E}^{n-1}\rightarrow[c]$ by $\chi'(v)=\chi(T(v))$.  Let $\{\textbf{u}_1,\ldots, \textbf{u}_{n-1}\}$ be the standard basis for $\mathbb{E}^{n-1}$, and let $\mathcal{U}=\{\textbf{0},\textbf{u}_1,\ldots,\textbf{u}_{n-1}\}$.

  Now we partition $\mathbb{E}^{n-1}$ into cosets.  Let $r\in\mathbb{R}^+$, and $r\mathbb{Z}:=\{rn:n\in\mathbb{Z}\}$.  We consider $\mathbb{E}^{n-1}/(r\mathbb{Z})^{n-1}$.  Each coset $\mathcal{C}$ is a translation of $(r\mathbb{Z})^{n-1}$ by a vector $\textbf{e}\in[0,r)^{n-1}$.  Again, we index the cosets by the vectors $\textbf{e}$.


  Each coset $\mathcal{C}_\textbf{e}$ is colored by $\chi'$, so by Gallai's Theorem on $\mathbb{Z}^n$ we may conclude that there exists some monochromatic subset of $\mathcal{C}_\textbf{e}$ homothetic to $r\mathcal{U}$.  This subset will have the form $ra\mathcal{U}+\textbf{d}+\textbf{e}$, where $a\in\mathbb{Z}^+$ and $\textbf{d}\in(r\mathbb{Z})^{n-1}$.  Now define
  \begin{equation*}
  \mathcal{S}_\textbf{e}=T(ra\mathcal{U}+\textbf{d}+\textbf{e})=\{T(\textbf{v}):\textbf{v}\in ra\mathcal{U}+\textbf{d}+\textbf{e}\}.
  \end{equation*}

  Again, $\mathcal{S}_\textbf{e}$ is a monochromatic.  We have $$\mathcal{S}_\textbf{e}=T(ra\mathcal{U}+\textbf{d}+\textbf{e})=ra\mathcal{S}+T(\textbf{d}+\textbf{e}),$$ so $\mathcal{S}_\textbf{e}$ is homothetic to $\mathcal{S}$ with dilation factor $ra$.


  It remains to show that, while many of these $\mathcal{S}_\textbf{e}$ may overlap, it is possible to find $2^{\aleph_0}$ of them that are pairwise disjoint.

  Since $A$ has rank $k$, the image under $T$ of $[0,r)^{n-1}$ contains a  $k$-dimensional parallelotope $\mathcal{P}$.  Let $\mathcal{Y}=\{m(\textbf{y}_i-\textbf{y}_j):m\in\mathbb{Z},\,\, 1\leq i\leq j<n\}$, and define an equivalence relation $\sim$ on $\mathcal{P}$ as follows: $$\textbf{p}\sim \textbf{q} \text{ if and only if }\textbf{p}-\textbf{q}\in\mathcal{Y}.$$
  Let $\mathcal{P}_0$ be a subset of $\mathcal{P}$ formed by choosing exactly one representative from each equivalence class.    As before, for each $\textbf{p}\in\mathcal{P}_0$ we choose $\textbf{e}_\textbf{p}\in[0,r)^{n-1}$ such that $T(\textbf{e}_\textbf{p})=\textbf{p}$, and let $\mathcal{E}=\{\textbf{e}_\textbf{p}:\textbf{p}\in\mathcal{P}_0\}$.  

  Now for each $\textbf{e}_\textbf{p}\in\mathcal{E}$ there is a coset $\mathcal{C}_{\textbf{e}_\textbf{p}}$ which contains a monochromatic set homothetic to $\mathcal{U}$ of the form $ra\mathcal{U}+\textbf{d}+\textbf{e}$, where $a\in\mathbb{Z}^+$ and $\textbf{d}\in(r\mathbb{Z})^{n-1}$.  Again, there are $a_0\in\mathbb{Z}^+$ and $\textbf{d}_0\in(r\mathbb{Z})^{n-1}$ such that for $2^{\aleph_0}$ of the cosets $\mathcal{C}_{\textbf{e}_\textbf{p}}$, the monochromatic subset of $\mathcal{C}_{\textbf{e}_\textbf{p}}$ homothetic to $\mathcal{U}$ has the form $ra_0\mathcal{U}+\textbf{d}_0+\textbf{e}_\textbf{p}$.

  Let $\mathcal{E}'$ be the subset of $\mathcal{E}$ consisting of all $\textbf{e}_\textbf{p}$ such that $ra_0\mathcal{U}+\textbf{d}_0+\textbf{e}_\textbf{p}$ is monochromatic.  Then by the same argument given at the end of the proof of  Theorem \ref{notline}, we have that $\mathcal{S}_{\textbf{e}_\textbf{p}}$ and $\mathcal{S}_{\textbf{e}_\textbf{q}}$ are disjoint for distinct $\textbf{e}_\textbf{p}$ and $\textbf{e}_\textbf{q} $.

  Now let $\mathcal{R}$ be a maximal subset of $\mathbb{R}^+$ that is linearly independent over $\mathbb{Z}$.  For each $r\in\mathcal{R}$ there is some $a_0$ such that there are $2^{\aleph_0}$ pairwise-disjoint subsets $\mathcal{S}_{\textbf{e}_\textbf{p}}$ with dilation factor $ra_0$.  Since $\mathcal{R}$ is linearly independent over $\mathbb{Z}$, no two dilation factors $ra_0$ coincide. Since $\mathcal{R}$ is maximal, $|\mathcal{R}|=2^{\aleph_0}$.  This concludes the proof.
\end{proof}

Finally, we note that the method of using  linear transformations
affords us an easy derivation of Gallai's Theorem on $\mathbb{E}^n$
from Gallai's Theorem on $\mathbb{Z}^n$.  We leave this derivation
to the reader, since  the necessary details are given above.

\section{Acknowledgments}

The author wishes to thank Jacob Manske for very many productive
conversations, and for critiquing an earlier draft of this paper.
This paper is dedicated to my son Thomas, who was born three days after
final acceptance.


\end{document}